\documentclass{llncs}
\usepackage{amsmath}
\usepackage{amssymb}
\usepackage{enumerate}
\usepackage{varioref}
\usepackage{charter,eulervm}
\usepackage{graphicx}
\usepackage{subfig,wrapfig}

\newcommand{\es}{\varnothing}

\title{{\sc Independent Sets in Edge-Clique Graphs}}

\author{
 Ching-Hao~Liu\inst{}
\and
 Ton~Kloks\inst{}
\and
 Sheung-Hung~Poon\inst{} 
} 
\institute{
 Department of Computer Science\\
 National Tsing Hua University, Taiwan\\
 {\tt chinghao.liu@gmail.com \hspace{.2cm} spoon@cs.nthu.edu.tw}
}

\begin{document}

\maketitle

\begin{abstract}
We first show that the independent set
problem on edge-clique graphs of cographs can be solved in $O(n^2)$ time. 
We then show that the independent 
set problem on edge-clique graphs of graphs without odd wheels 
is NP-complete. 
We present a PTAS for planar graphs and show 
that the problem is polynomial 
for planar graphs without triangle separators.
Lastly, we show that edge-clique graphs of cocktail party graphs 
have unbounded rankwidth. 
\end{abstract}

\section{Introduction}

Let $G=(V,E)$ be an undirected graph with vertex set $V$ and 
edge set $E$. A clique is a complete subgraph of $G$. 

\begin{definition}
The edge-clique graph $K_e(G)$ of a graph $G$ has
the edges of $G$ as its vertices and two vertices of 
$K_e(G)$ are adjacent when the corresponding edges in $G$ are 
contained in a clique. 
\end{definition}

\smallskip

Edge-clique graphs were introduced and studied by 
Albertson and Collins%
~\cite{kn:albertson}.
Two characterizations of edge-clique graphs 
have been presented by~\cite{kn:cerioli4,kn:chartrand},
but the problem whether there is any 
polynomial-time recognition algorithm for edge-clique graphs remains open.
Some results of edge-clique graphs can be 
found 
in~\cite{kn:albertson,kn:calamoneri,kn:cerioli2,kn:cerioli4,%
kn:cerioli3,kn:chartrand,kn:kong1,kn:mckee,kn:prisner,%
kn:raychaudhuri1,kn:raychaudhuri2}.
Notice that edge-clique graphs were
implicitly first used by Kou et al~\cite{kn:kou}.
In the following we introduce the independent set problem.

\smallskip

An independent set in a graph $G$ is a nonempty set of vertices $A$
such that there is no edge between any pair of vertices in $A$.
The independent set problem asks the maximal cardinality of independent sets, 
which is called independent set number and denoted by $\alpha(G)$.
In general graphs, this problem is known to be NP-hard
and it remains NP-hard even for triangle-free graphs~\cite{kn:poljak}
and planar graphs of degree at most three~\cite{kn:garey}.
In some classes of graphs, for example claw-free graphs~\cite{kn:sbihi} 
and perfect graphs~\cite{kn:grotschel}, 
this problem can be solved in polynomial time.

\smallskip

A subset $A$ of edges in a graph $G$ is independent 
if it induces an independent set in $K_e(G)$,
that is, no two edges of $A$ are contained in a clique of $G$.
We denote the maximal cardinality of independent sets of edges of $G$ by
$\alpha^{\prime}(G)=\alpha(K_e(G))$.
Notice that, for any graph $G$, the clique number of 
its edge-clique graph satisfies 
\[\omega(K_e(G))=\binom{\omega(G)}{2}.\] 
For the independent set number $\alpha(K_e(G))$ there is 
no such relation. 
For example, when $G$ has no triangles then $K_e(G)$ is 
an independent set and the independent 
set problem in triangle-free graphs is NP-complete.

\smallskip

There is another problem relating to independent sets in edge-clique graphs, 
which is called the edge-clique cover problem.
An edge-clique covering of $G$ is a family of 
complete subgraphs such that 
each edge of $G$ is in at least one member of the family. 
The edge-clique cover problem asks the minimal cardinality of such a family,
which is called the edge-clique covering number and denoted by $\theta_e(G)$.

\smallskip 

The problem of deciding if $\theta_e(G) \leq k$, 
for a given natural number $k$, is NP-complete~\cite{kn:holyer,kn:kou,kn:orlin}. 
The problem remains NP-complete when restricted to graphs 
with maximum degree at most six~\cite{kn:hoover}.  
Hoover~\cite{kn:hoover} gave a polynomial time algorithm for 
graphs with maximum degree at most five. 
For graphs with 
maximum degree less than five, this was already done 
by Pullman~\cite{kn:pullman}. 
Also for linegraphs the problem can be solved in 
polynomial time~\cite{kn:orlin,kn:pullman}. 
In~\cite{kn:kou} it was shown that approximating the 
clique covering number within a constant factor 
smaller than two remains NP-complete. 
The independence number of $K_e(G)$ equals $\theta_e(G)$ for 
graphs $G$ that are chordal. 
For interval graphs 
the edge-clique covering number $\theta_e(G)$ 
equals the number of maximal cliques~\cite{kn:scheinerman}.

\smallskip

We organized this paper as follows.
In Section~2 and Section~3 we show that 
the independent set problem on edge-clique graphs  
is NP-hard in general and remains so for edge-clique graphs of graphs without 
odd wheels.
We show in Sections~2.1 and~2.2 
that the problem can be solved in $O(n^2)$ time for cographs  
and in polynomial time for distance-hereditary graphs. 
In Section~4, we present a PTAS for planar graphs
and show that the problem is polynomial for planar graphs without triangle separator.

\smallskip 

Notice that cographs have rankwidth one. 
If edge-clique graphs of cographs $G$ would have
bounded rankwidth, then computing $\alpha^{\prime}(G)$ would be 
easy,  because computing the independence number is polynomial 
for graphs of bounded rankwidth.
However, we show in Section~5 that this is not the case, not even for 
edge-clique graphs of cocktail parties. 
Finally, we conclude in Section~6.

\section{Algorithms for cographs}

In this section we begin with a NP-hardness proof
for the computation of $\alpha^{\prime}(G)$ for arbitrary graphs $G$.
Then we provide algorithms for 
the computation of independence number
of edge-clique graphs of cographs and distance-hereditary graphs.

\smallskip

The following lemma shows that the independent set problem 
in $K_e(G)$ can be reduced to the independent set problem in $G$. 

\begin{lemma}
\label{alpha' NP-hard}
The computation of $\alpha^{\prime}(G)$ 
for arbitrary graphs $G$ is NP-hard.
\end{lemma}
\begin{proof}
Let $G$ be an arbitrary graph.
Construct a graph $H$ as follows. At every edge in $G$ add two
simplicial vertices\footnote{An additional argument would show that 
adding one simplicial per edge is sufficient to prove NP-hardness.},
both adjacent to the two endvertices of the edge.
Add one extra vertex $x$ adjacent to all vertices of $G$.
Let $H$ be the graph constructed in this way. 

\smallskip 

\noindent
Notice that a maximum set of independent edges does not
contain any edge of $G$ since it would be better to replace
such an edge by two edges incident with the two
simplicial vertices at this edge.
Also notice that a set of independent edges incident with $x$
corresponds with an independent set of vertices in $G$.
Hence 
\[\alpha^{\prime}(H) = 2m + \alpha(G),\] 
where
$m$ is the number of edges of $G$.
\qed\end{proof}

\subsection{Cographs}

A cograph is a graph without induced $P_4$. It is well-known that 
a graph is a cograph if and only if every induced subgraph with 
at least two vertices is either a join or a union of two 
smaller cographs. It follows that a cograph $G$ has a decomposition 
tree $(T,f)$ where $T$ is a rooted binary tree and $f$ is a bijection 
from the vertices of $G$ to the leaves of $T$. Each internal node 
of $T$, including the root, is labeled as $\otimes$ or $\oplus$. 
The $\otimes$-node joins the two subgraphs mapped to the left 
and right subtree. The $\oplus$ unions the two subgraphs. 
When $G$ is a cograph then a decomposition tree as described 
above can be obtained in linear time~\cite{kn:corneil}. 
  
\begin{lemma}
\label{adjacent edges in join}
Let $G$ be a cograph. Assume that $G$ is the join of two smaller 
cographs $G_1$ and $G_2$. Then any edge in $G_1$ is adjacent in 
$K_e(G)$ to any edge in $G_2$. 
\end{lemma}

\smallskip 

For a vertex $x$, let $d^{\prime}(x)$ be the independence 
number of the subgraph of $G$ induced by $N(x)$, that is,  
\begin{equation}
d^{\prime}(x)=\alpha(G[N(x)]).
\end{equation}  

\begin{lemma}
\label{lem cograph}
Let $G$ be a cograph. Then 
\begin{equation}
\label{eqn0}
\alpha^{\prime}(G)=\max \;\{\;\sum_{x \in W} d^{\prime}(x)\;|\; 
\text{$W$ is an independent set in $G$}\;\}. 
\end{equation}
\end{lemma}

\begin{proof}
Cographs are characterized by the fact that every induced subgraph has a twin.
Let $x$ be a false twin of a vertex $y$ in $G$. 
Let $A$ be a maximum independent set of edges in $G$. 
Let $A(x)$ and $A(y)$ be the  
sets of edges in $A$ that are incident with $x$ and $y$. 
Assume that $|A(x)| \geq |A(y)|$. Let $\Omega(x)$ 
be the set of endvertices in $N(x)$ of edges in $A(x)$.  
then we may replace the set $A(y)$ 
with the set 
\[\{\;\{y,z\}\;|\; z \in \Omega(x)\;\}.\] 
The cardinality of the new set is at least as large as $|A|$. 
Notice that, for any maximal independent set $Q$ in $G$, either 
$\{x,y\} \subseteq Q$ or $\{x,y\} \cap Q = \es$. 
By induction on the number of 
vertices in $G$, Equation~(\ref{eqn0}) is valid. 

\smallskip 

\noindent
Let $x$ be a true twin of a vertex $y$ in $G$. 
Let $A$ be a maximum independent set of edges in $G$ and 
let $A(x)$ and $A(y)$ be the sets of  
edges in $A$ that 
are incident with $x$ and $y$. 
If $\{x,y\} \in A$ then $A(x)=A(y)=\{\{x,y\}\}$. 

\smallskip 

\noindent
Now assume that $\{x,y\} \notin A$. 
Endvertices in $N(x)$ of edges in $A(x)$ and $A(y)$ 
are not adjacent nor do they coincide. Replace $A(x)$ with 
\[\{\;\{x,z\}\;|\; \{x,z\} \in A(x) \quad\text{or}\quad 
\{y,z\} \in A(y) \;\}\] 
and set $A(y)=\es$. Then the new set of edges is independent 
and has the same cardinality as $A$.   

\smallskip 

\noindent
Let $Q$ be an independent set in $G$. At most one of 
$x$ and $y$ is in $Q$. 
By induction on the number of vertices in $G$,
the validity of Equation~(\ref{eqn0}) 
is easily checked.  
\qed\end{proof}

\begin{remark}
Notice that the righthand side of Equation~(\ref{eqn0})
is a lowerbound for $\alpha^{\prime}(G)$ for any graph $G$.
In this paper we find that there is an equality for some 
classes of graphs. Notice that, for example, equality does 
not hold for $C_5$. 
\end{remark}

\smallskip

\begin{theorem}
\label{thm cograph}
When $G$ is a cograph then 
$\alpha^{\prime}(G)$ satisfies Equation~{\rm (\ref{eqn0})}.
This value can be computed in $O(n^2)$ time.
\end{theorem}
\begin{proof}
We proved that $\alpha^{\prime}(G)$ 
satisfies Equation~{\rm (\ref{eqn0})}
in Lemma~\ref{lem cograph}.
Here we show that $\alpha^\prime(G)$ 
can be computed in $O(n^2)$ time.

\smallskip 

\noindent 
The algorithm first computes a decomposition tree $(T,f)$ 
for $G$ in linear time. For each node $p$ of $T$ 
let $G_p$ be the subgraph induced by the set of 
vertices that are mapped to leaves in the subtree rooted at $p$.   
Notice that the independence 
number of each $G_p$ can be computed in linear time. 

\smallskip 

\noindent
Compute $d^{\prime}(x)$ for $x \in V$ as follows. 
Assume $G$ is the union of $G_1=(V_1,E_1)$ and $G_2=(V_2,E_2)$. 
For $i \in \{1,2\}$, 
let $d_i^{\prime}(x)=\alpha(G_i[N(x)])$ for vertices $x$ in $G_i$. 
Then 
\begin{equation}
d^{\prime}(x)=
\begin{cases}
d^{\prime}_1(x) \quad\text{if $x \in V_1$} \\
d^{\prime}_2(x) \quad\text{if $x \in V_2$.}
\end{cases}
\end{equation}

\smallskip 

\noindent 
Assume that $G$ is the join of $G_1$ and $G_2$. 
Then 
\begin{equation}
d^{\prime}(x)=
\begin{cases}
\max\;\{\;d^{\prime}_1(x),\;\alpha(G_2)\;\} \quad\text{if $x \in V_1$}\\
\max\;\{\;d^{\prime}_2(x),\;\alpha(G_1)\;\} \quad\text{if $x \in V_2$}. 
\end{cases}
\end{equation}

\smallskip 

\noindent 
Evaluate $\alpha^{\prime}(G)$ as follows. 
Assume that $G$ is the union of $G_1$ and $G_2$. 
Then 
\begin{equation}
\alpha^{\prime}(G)=\alpha^{\prime}(G_1)+\alpha^{\prime}(G_2). 
\end{equation}

\smallskip 

\noindent
Assume that $G$ is the join of $G_1$ and $G_2$. 
For $i \in \{1,2\}$, let 
\begin{equation}
\alpha^{\prime}_i 
=\max \;\{\;\sum_{x \in W}\; d^{\prime}(x)\;|\; 
\text{$W$ is an independent set in $G_i$}\;\}.
\end{equation}
Then 
\begin{equation}
\alpha^{\prime}(G)=\max \; \{\;\alpha^{\prime}_1,\;\alpha^{\prime}_2\;\}.
\end{equation}

\smallskip 

\noindent
This completes the proof.
\qed\end{proof}

\begin{remark}
For any graph $G$ whose edge-clique graph is perfect
the intersection number of $G$ equals the fractional 
intersection number of $G$~\cite[Theorem 4.1]{kn:kong1}. 
It is easy to see that for cographs $G$, $K_e(G)$ is not 
necessarily perfect. 
For example, when $G$ is the join of 
$P_3$ and $C_4$, then $K_e(G)$ 
contain a $C_5$ as an induced subgraph.\footnote{Let $P_3=\{1,2,3\}$ and $C_4=\{4,5,6,7\}$.
There is a $C_5$ with vertices 
\[\{\{1,2\},\{3,4\},\{3,7\},\{4,5\},\{6,7\}\}.\]\vspace{-5mm}}
\end{remark}     

\subsection{Distance-hereditary graphs}

A graph $G$ is distance-hereditary if the distance 
between any two nonadjacent vertices, 
in any connected induced subgraph of $G$, 
is the same as their distance in the $G$~\cite{kn:howorka}. 
Bandelt and Mulder obtained the following 
characterization of distance-hereditary graphs. 

\begin{lemma}[\cite{kn:bandelt}]
\label{DH characterize}
A graph is distance-hereditary if and only if every induced subgraph 
has an isolated vertex, a pendant vertex or a twin.
\end{lemma}
The papers~\cite{kn:bandelt} and~\cite{kn:howorka} 
also contain characterizations of 
distance-hereditary graphs in terms of forbidden induced 
subgraphs. 

\bigskip 

Let $G$ be distance-hereditary. 
Then $\alpha^{\prime}(G)$ does not necessarily satisfy  
Equation~{\rm (\ref{eqn0})}. A counterexample is as 
follows. Take two 4-wheels and join the two centers 
by an edge. The maximum edge-independent set 
consists of the edges of the two 4-cycles and the edge 
joining the two centers. However, there is no 
independent set such that~\eqref{eqn0} gives this value. 

\section{NP-completeness for graphs without odd wheels}

A wheel $W_n$ is a graph consisting of a cycle $C_n$ and one 
additional vertex adjacent to all vertices in the cycle. 
The universal vertex of $W_n$ is called the hub. It is unique unless 
$W_n=K_4$. The edges incident with the hub are called the 
spokes of the wheel. The cycle is called the rim of the wheel. 
A wheel is odd if the number of vertices in the cycle is odd. 

\smallskip 

Lakshmanan, Bujt\'as and Tuza investigate the class of graphs 
without odd wheels in~\cite{kn:lakshmanan}. 
They prove that Tuza's conjecture holds true for 
this class of graphs (see also~\cite{kn:haxell}). 
Notice that a graph $G$ has no odd wheel if and only if 
every neighborhood in $G$ induces a bipartite graph. It follows 
that $\omega(G) \leq 3$. Obviously, the class of graphs 
without odd wheels is closed under taking subgraphs. 
Notice that, when $G$ has no odd wheel then every neighborhood 
in $K_e(G)$ is either empty or a matching. Furthermore, 
it is easy to see that $K_e(G)$ contains no diamond (every 
edge is in exactly one triangle), no $C_5$ and no  
odd antihole. 

\smallskip 

For graphs $G$ without odd wheels 
$K_e(G)$ coincides 
with the anti-Gallai graphs introduced by Le~\cite{kn:le},
since $\omega(G) \leq 3$.
For general anti-Gallai graphs the computation of the clique  
number and chromatic number 
are NP-complete.

Let us mention that the recognition of anti-Gallai graphs 
is NP-complete. Even when each edge in $G$ is in exactly one triangle, 
the problem to decide if $G$ is an anti-Gallai graph is 
NP-complete~\cite[Corollary 5.2]{kn:anand}. 
The recognition of edge-clique graphs of graphs without odd 
wheels is, as far as we know, open. Let us also mention 
that the edge-clique graphs 
of graphs without odd wheels are clique graphs~\cite{kn:cerioli2}. 
The recognition of 
clique graphs of general graphs is NP-complete~\cite{kn:alcon}.  

\begin{theorem}
\label{NP_c}
The computation of $\alpha^{\prime}(G)$ is NP-complete 
for graphs $G$ without odd wheels. 
\end{theorem}
\begin{proof}
We reduce 3-SAT to the vertex cover problem in 
edge-clique graphs of graphs without odd wheels. 

\smallskip 

\noindent
Let $H \simeq cp(3)$, ie, $L(K_4)$.
See Fig.~\ref{fig:H_KeH}{\sc (A)}.
Let $S$ be a 3-sun as depicted in Fig.~\ref{fig:H_KeH}{\sc (B)}.
The graph $H$ is obtained from $S$ by adding 
three edges between pairs of vertices of 
degree two in $S$.%
\footnote{In~\cite[Theorem~14]{kn:lakshmanan2} the authors 
prove that every maximal clique in $K_e(G)$ contains a 
simplicial vertex if and only if $G$ does not contain,  
as an induced subgraph, $K_4$  
nor a 3-sun with 0, 1, 2 or 3 edges connecting the 
vertices of degree two.}
Call the three vertices of degree four in $S$, the `inner 
triangle' of $H$ and call the remaining three vertices of $H$ the 
`outer triangle.' 

\begin{figure}[t]
  \centering
  \parbox{0.7\textwidth}{
  \includegraphics[width=0.69\textwidth]{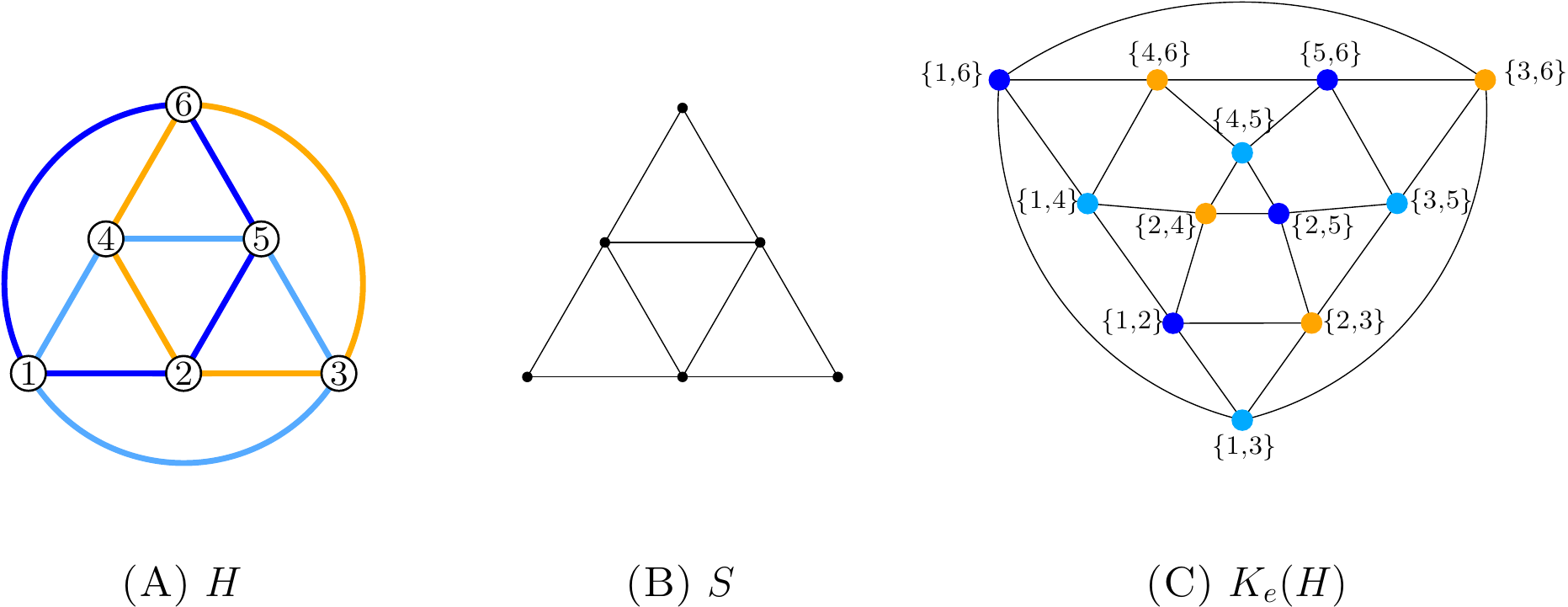}
  \caption{This figure shows $H$, $S$ and $K_e(H)$.
 In $H$, vertices 2, 4, 5 induce the inner triangle
 and the remaining three vertices induce the outer triangle.
 The three colors for the edges of $H$ indicate the partition
 of three maximum independent sets of edges.
  }
  \label{fig:H_KeH}}
  \qquad
  \begin{minipage}{2.7cm}
  \includegraphics[width=1.0\textwidth]{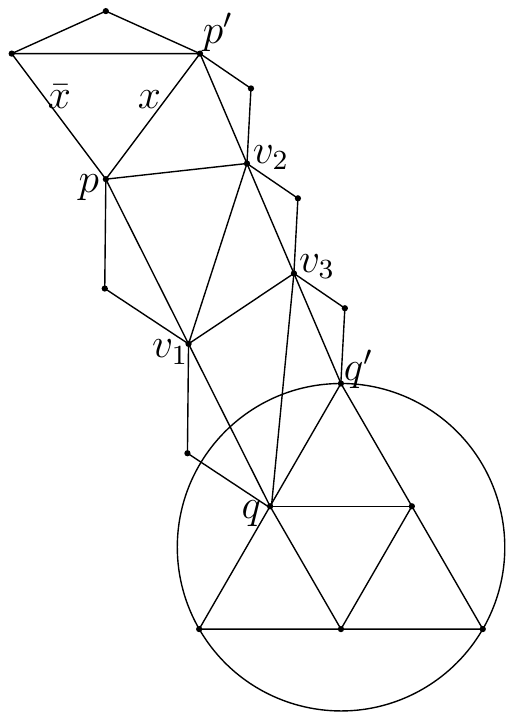}
  \caption{This figure shows the construction of a link 
  between a variable gadget and a clause gadget.
  }
  \label{fig:link}  
  \end{minipage}
  \vspace{-5mm}
\end{figure}

\noindent
Notice that $H$ has 3 maximum independent sets of 
edges. 
Each maximum independent set of edges is an induced $C_4$ 
consisting of one 
edge from the inner triangle, one edge from the outer triangle, 
and two edges between the two triangles. The three independent sets 
partition the edges of $H$. 

\noindent
The six edges of $H$ between the inner and outer triangle form 
a 6-cycle in $K_e(H)$. Let $F$ denote this set of edges in $H$.  

\smallskip 

\noindent
For each clause $(x_i \; \vee \; x_j \; \vee \; x_k)$
take one copy of $H$.   
Take an independent set of three edges contained in $F$ 
and label these with $x_i$, $x_j$ and $x_k$. 

\smallskip 

\noindent
For each variable $x$ take a triangle. Label one edge 
of the triangle with the literal $x$ and one other edge of 
the triangle with its negation $\Bar{x}$. 
Then add a simplicial vertex to the unlabeled edge
of the triangle.

\smallskip 

\noindent 
Construct links between variable gadgets and clause gadgets 
as follows. 
Let $(x_i\;\vee\;x_j\;\vee\;x_k)$ be a clause. 
There are four steps.
First add a 2-chain $p v_1 q$ from an endpoint $p$ 
of the edge labeled $x_i$ in the variable gadget
to one endpoint $q$ of the edge $x_i$ in the clause gadget.
See Fig.~\ref{fig:link}.
Similarly, add a 3-chain $p' v_2 v_3 q'$ between 
the other two endpoints, where $p'$ is in the 
variable gadget and $q'$ is in the clause gadget.
Then add four edges $p v_2$, $v_2 v_1$, $v_1 v_3$ and $v_3 q$
to let the rectangle $pqq'p'$ become a triangle strip with five triangles.
Finally, add a simplicial vertex to 
each of the edges in the 2-chain and the 3-chain.
We construct links for the other two literals in the clause 
in the same manner. 

\smallskip

\noindent 
Let $G$ be the graph constructed in this manner. 
Consider the triangles $T$ which are 
the edge-clique graphs of the triangles in $G$
containing the newly-added simplicial vertices.
Notice that simplicial vertices of a graph may be removed without 
changing the complexity of the vertex cover problem and so 
the vertices of $T$ can be removed without changing the complexity.
Let $K$ be the graph obtained from $K_e(G)$ by removing 
the edges of these triangles. 

\smallskip 

\noindent
Let $L$ be the number of variables, 
let $M$ be the number of clauses in the 3-SAT formula. 
Assume that there is a satisfying assignment. Then choose the vertices in 
$K$  
corresponding to literals that 
are {\sc true} in the vertex cover. 
The variable gadgets need $L$ vertices and the 
links need $6M$ vertices in the vertex cover.
Since this assignment is satisfying, we need at most 
$8M$ vertices to cover the remaining edges in the 
clause structures, since the outgoing edge from each literal 
which is {\sc true} is covered. Thus there is a vertex 
cover of $K_e(G)$ with 
$L+14M$ vertices. 

\smallskip 

\noindent
Assume that $K_e(G)$  
has a vertex cover with $L+14M$ vertices. At least 
$L$ vertices in $K$ are covering the edges in the variable gadgets
and at least $6M$ vertices in $K$ are covering the edges in the links.
The other $8M$ vertices of $K$ are covering the 
edges in the clause gadgets. 
Take the literals of the variable gadgets 
that are in the vertex cover as an assignment 
for the formula. Each clause gadget must have one literal vertex 
of which the outgoing edge is covered. Therefore, the assignment is 
satisfying. 

\medskip 

\noindent
This proves the theorem. 
\qed\end{proof}

\section{Algorithms for planar graphs}

In this section we obtain 
a PTAS for a maximum independent set of edges in planar graphs. 

\begin{lemma}
\label{bounded tw}
Let $k \in \mathbb{N}$. There exists a linear-time algorithm 
which computes a maximum independent set of edges for graphs 
of treewidth at most $k$.
\end{lemma}
\begin{proof}
Notice that the problem can be formulated in 
monadic second-order logic. The claim follows from 
Courcelle's theorem~\cite{kn:courcelle}. 
\qed\end{proof}

Consider a plane graph $G$. 
If all vertices lie on the outerface, then $G$ is 1-outerplanar. 
Otherwise, $G$ is $k$-outerplanar if the removal of all vertices 
of the outerface results in a $(k-1)$-outerplanar graph. 
A graph is $k$-outerplanar if it has a $k$-outerplanar embedding 
in the plane~\cite{kn:baker}. 

\begin{theorem}[\cite{kn:bodlaender}]
The treewidth of a $k$-outerplanar graph is at most $3k-1$. 
\end{theorem}

\begin{theorem}
\label{thm PTAS planar}
There exists a polynomial-time approximation scheme 
for a maximum independent set of edges in planar graphs. 
\end{theorem}
We use Baker's technique~\cite{kn:baker} in the proof of this theorem
and we moved the proof to Appendix~\ref{Appendix C}. 

\begin{remark}
It is easy to see that the treewidth algorithm, for graphs of 
treewidth $k$, can be implemented to run in $O(2^{O(k)} n)$ time. 
Since the treewidth of planar graphs is bounded by $O(\sqrt{n})$, 
this gives an exact algorithm for a maximum independent set 
of edges which runs in $O(2^{O(\sqrt{n})})$ time. 
\end{remark}

\begin{remark}
Blanchette, Kim and Vetta 
prove in~\cite{kn:blanchette} that there is a 
PTAS for edge-clique cover of planar graphs. 
\end{remark}

\subsection{Planar graphs without triangle separator}

A triangle separator (of a connected graph) 
is a triangle the removal of which disconnect the graph.
We compute $\alpha^{\prime}(G)$ for planar graphs $G$ 
without triangle separator
in the following theorem.

\begin{theorem}
There exists an $O(n^{3/2})$ algorithm that computes a 
maximum independent set of edges in planar graphs without 
triangle separator. 
\end{theorem}
\begin{proof}
Let $G$ be an embedding of a planar graph without triangle 
separator. We may assume that $G \not\simeq K_4$.  
In all the faces that have length more than three add 
a new vertex and make it adjacent to all vertices in the face. 
Let $G^{\prime}$ be this graph. Give the edges of $G$ a weight 1 
and give the new edges a weight 0.  

\smallskip 

\noindent 
Let $H$ be the dual of $G^{\prime}$. In $H$ the weight of an edge 
is the weight of the edge in $G^{\prime}$ that it crosses. 
We claim that a 
solution is obtained by computing a maximum weight 
matching in $H$~\cite{kn:edmonds}. 

\smallskip 

\noindent
Let $M^{d}$ be a matching in $H$. We may assume that $M^{d}$ contains 
no edges of weight 0. Then every edge of $M^{d}$ crosses an edge 
of $G$. We claim that this is an independent set $M$ 
of edges in $G$. 
Since $M^{d}$ is a matching no two triangular 
faces of $G^{\prime}$ incident 
with different edges of $M^{d}$ coincide. Assume that two edges 
of $M$ lie in a triangle $T$ of $G$. Then $T$ cannot be a face 
of $G^{\prime}$ 
since $M^{d}$ is a matching. But then $T$ is a separator which 
contradicts our assumption. 

\smallskip 

\noindent 
Let $M$ be an independent set of edges in $G$ and let 
$M^{d}$ be the corresponding edges of $H$. 
Since $M$ is an independent set of edges no two edges of $M$ 
lie in a triangle. Assume that $M^{d}$ is not a matching, 
and let $f$ be a common face of $G^{\prime}$ of 
two edges in $M^{d}$. Then $f$ is contained in a 
face of $G$ which is not a triangle. But then the edges 
of $M$ are the same. 
\qed\end{proof}
           
\section{Rankwidth of edge-clique graphs of cocktail parties}

In this section we show that the related problem to 
compute $\theta_e(G)$ is probably much harder than the 
independence number. We show that $\theta_e(G)$ relates 
to well-known open problems for cocktail party graphs $G$. 
Cocktail parties are a proper subclass of cographs. 

\begin{definition}
The cocktail party graph $cp(n)$ is the complement of a 
matching with $2n$ vertices. 
\end{definition}
Gy\'arf\'as~\cite{kn:gyarfas} showed
an interesting lowerbound in the following theorem.
Here, two vertices $x$ and $y$ are {\em equivalent\/} if 
they are adjacent and have the same closed neighborhood. 

\begin{theorem}
\label{gyarfas}
If a graph $G$ has $n$ vertices and contains neither isolated 
nor equivalent vertices then $\theta_e(G) \geq \log_2(n+1)$. 
\end{theorem}

Notice that a cocktail party graph has no 
equivalent vertices. Thus, by Theorem~\ref{gyarfas}, 
\[\theta_e(cp(n)) \geq log_2(2n+1).\] 


For the cocktail party graph 
an exact formula for $\theta_e(cp(n))$ 
appears in~\cite{kn:gregory}. 
In that paper Gregory and Pullman (see also~\cite{kn:gargano,kn:korner}) 
prove that 
\[\lim_{n \rightarrow \infty} \; \frac{\theta_e(cp(n))}{\log_2(n)}=1.\] 




For a graph $G$ we denote the vertex-clique cover number of 
$G$ by $\kappa(G)$. 
Notice that, for a graph $G$,  
\[\theta_e(G)=\kappa(K_e(G)).\] 


Albertson and Collins mention the following result (due to 
Shearer)~\cite{kn:albertson} for the graphs $K_e^r(cp(n))$, 
defined recursively by $K_e^r(cp(n))=K_e(K_e^{r-1}(cp(n)))$.
\[\alpha(K_e^r(cp(n))) \leq 3\cdot (2^r)!\] 
Thus, for $r=1$, $\alpha(K_e(cp(n))) \leq 6$. 
However, the following is easily checked (see Lemma~\ref{lem cograph}). 

\begin{lemma}
\label{bound alpha}
For $n \geq 2$ 
\[\alpha(K_e(cp(n))) =4.\]
\end{lemma}


\begin{definition}
A class of graphs $\mathcal{G}$ is $\chi$-bounded if there 
exists a function $f$ such that for every graph $G \in \mathcal{G}$, 
\[\chi(G) \leq f(\omega(G)).\]
\end{definition}
Dvo\v{r}\'ak and Kr\'al' prove that the class of graphs 
with rankwidth at most $k$ is $\chi$-bounded~\cite{kn:dvorak}. 

\begin{theorem}
The class of edge-clique graphs of cocktail parties has 
unbounded rankwidth. 
\end{theorem}
\begin{proof}
It is easy to see that the rankwidth of any graph 
is at most one more than the rankwidth of its complement~\cite{kn:oum}. 
Assume that there is a constant $k$ such that 
the rankwidth of $K_e(G)$ is at most $k$ whenever $G$ is a 
cocktail party graph. Let 
\[\mathcal{K}=\{\; \overline{K_e(G)} \;|\; G \simeq cp(n), 
\;n \in \mathbb{N}\;\}.\] 
Then the rankwidth of graphs in $\mathcal{K}$ is uniformly bounded 
by $k+1$. By the result of Dvo\v{r}\'ak and Kr\'al', there exists 
a function $f$ such that 
\[log_2(2n+1) \leq \theta_e(G)=\kappa(K_e(G)) \leq f(\alpha(K_e(G))) = 4f\] 
for every cocktail party graph $G$. This contradicts 
Lemma~\ref{bound alpha} and Theorem~\ref{gyarfas}. 
\qed\end{proof}

\section{Conclusion and discussions}

In this paper we show polynomial-time algorithms 
and NP-completeness proofs
for the independence number of edge-clique 
graphs of some classes of graphs.
In particular we show that for cographs 
and distance-hereditary graphs
the independence number of edge-clique graphs 
satisfies Formula~{\rm (\ref{eqn0})}.
The results lead us ask new questions that 
for which other classes of graphs the formula also holds true
and under what conditions the formula is true.
For edge-clique cover problem
we make a conjecture as follows.
\begin{conjecture}
The edge-clique cover problem is NP-complete for cographs. 
\end{conjecture}
We move the motivation of this conjecture to Appendix~\ref{Appendix D}.

\smallskip

Finally, trivially perfect graphs is the subclass of cographs 
that are chordal.
A graph is trivially perfect if it does not contain $C_4$ nor 
$P_4$ as an induced subgraph. For this class of graphs we 
have equality:  

\begin{theorem}
If a graph $G$ is connected and trivially perfect then 
$\alpha^{\prime}(G)=\theta_e(G)$. 
\end{theorem}
\begin{proof}
If $G$ contains only one vertex then $\alpha^{\prime}(G)=\theta_e(G)=0$. 
Assume that $G$ contains more than one vertex.
Since $G$ is connected then $\alpha^{\prime}(G) \geq \alpha(G)$.
When a graph $G$ is trivially perfect then 
the independence number is equal 
to the number of maximal cliques in $G$. 
Thus $\alpha(G)= \theta_e(G) \geq \alpha^{\prime}(G)$.
Finally, we conclude $\alpha^{\prime}(G)=\theta_e(G)$.
\qed\end{proof}

\clearpage

\appendix

\section{Proof of Theorem~\ref{thm PTAS planar}}
\label{Appendix C}

\begin{theorem}
There exists a polynomial-time approximation scheme 
for a maximum independent set of edges in planar graphs. 
\end{theorem}
\begin{proof}
We use Baker's technique~\cite{kn:baker}. 
Consider a plane embedding of $G$. Let $L_0$ be the set of vertices 
in the outerface. Remove the vertices of $L_0$ from $G$ 
and let $L_1$ be the new outerface. Continuing this process partitions 
the vertices into layers. 

\smallskip 

\noindent
Fix an integer $k$. 
For $i \in \{0,\dots,k-1\}$ and $j\in \mathbb{N} \cup \{-1,0\}$, let 
$G_{ij}$ be the subgraph of $G$ induced by the vertices 
in layers $L_t$ where 
\[t \in \{\;jk+i+1,\;jk+i+2,\;\dots,\; jk+i+(k-1)\;\}~~~\text{and}~~~t \geq 0.\] 
Notice that the outerplanarity of each $G_{ij}$ is at most $k$. 
By lemma~\ref{bounded tw} there exists a linear-time algorithm 
which computes a maximum independent set of edges $A_{ij}$ in $G_{ij}$. 

\smallskip 

\noindent
Every edge in $G$ connects two vertices in 
either the same layer or in two adjacent layers. For fixed $i$, 
the consecutive graphs $G_{ij}$ skip two layers. So 
$\cup_{j}A_{ij}$ is an independent set of 
edges in $G$. 

\smallskip 

\noindent
Let $A$ be a maximum independent set of edges in $G$. 
Sum over all graphs $G_{ij}$ the edges of $A$ that are contained 
in $G_{ij}$. Then every edge of $A$ is counted $k-2$ times. 
Also, since $A_{ij}$ is an exact solution for $G_{ij}$ this is 
at least as big as the number of edges induced by $A$ in $G_{ij}$. 
Therefore, the sum of $A_{ij}$ is at least $k-2$ times the optimum. 
If we take the maximum over $i\in\{0,\dots,k-1\}$, we find an 
approximation of size at least $1-\frac{2}{k}$ times the optimum. 

\smallskip 

\noindent
This proves the theorem.
\qed\end{proof}

\section{Conjecture on edge-clique cover problem for cographs}
\label{Appendix D}

According to Theorem~\ref{gyarfas}
Gy\'arf\'as~\cite{kn:gyarfas} result implies that the edge-clique 
cover problem is fixed-parameter tractable (see also~\cite{kn:gramm}). 

\smallskip 

Cygan et al~\cite{kn:cygan} show that, under the assumption of the 
exponential time hypothesis, there is 
no polynomial-time algorithm which reduces the 
parameterized problem $(\theta_e(G),k)$ to a kernel of size bounded 
by $2^{o(k)}$. In their proof the authors make use of the fact that 
$\theta_e(cp(2^{\ell}))$ is a [sic] 
``hard instance for the edge-clique cover problem, at least from a 
point of view of the currently known algorithms.''    
Note that, in contrast, the parameterized edge-clique partition 
problem can be reduced to a kernel with at most $k^2$ 
vertices~\cite{kn:mujuni}. (Mujuni and Rosamond also mention that 
the edge-clique cover problem probably has no polynomial kernel.)
These observations lead us 
to investigate edge-clique graphs of cocktail parties
in Section~5. 

\smallskip

Let $K_n^m$ denote the complete multipartite graph with $m$ partite sets 
each having $n$ vertices. Obviously, $K_n^m$ is a cograph with $n\cdot m$ 
vertices. 

\smallskip

\begin{theorem}[\cite{kn:park}]
Assume that 
\[3 \leq m \leq n+1.\] 
Then $\theta_e(K_n^m)=n^2$ if and only if there exists  
a collection of at least $m-2$ pairwise orthogonal Latin squares 
of order $n$. 
\end{theorem}
Notice that, if there exists an edge-clique cover of $K_n^m$ 
with $n^2$ cliques, then these cliques are mutually edge-disjoint. 
Finding the maximal number of mutually orthogonal 
Latin squares of order $n$ is a renowned open problem. 
The problem has a wide field 
of applications, eg in combinatorics, designs 
of experiments, group theory and quantum informatics.   

\smallskip

Unless $n$ is a prime power, 
the maximal number of MOLS is known for only a few 
orders. We briefly mention a few results.
Let $f(n)$ denote the 
maximal number of MOLS of order $n$. The well-known `Euler-spoiler'  
shows that $f(n)=1$ only for 
$n=2$ and $n=6$. 
Also, $f(n) \leq n-1$ for all $n>1$, 
and Chowla, Erd\"os and Straus~\cite{kn:chowla} 
show that 
\[\lim_{n \rightarrow \infty} f(n) = \infty.\] 

\smallskip

Define 
\[n_r = \max\;\{\;n\;|\; f(n) < r\;\}.\] 
A lowerbound for the speed at which $f(n)$ grows was obtained 
by Wilson, who showed that $n_r < r^{17}$ when 
$r$ is sufficiently large~\cite{kn:wilson}. 
Better bounds for $n_r$, for some specific values of $r$, were 
obtained by various authors (see eg~\cite{kn:brouwer}). 

\smallskip

See eg ~\cite{kn:feifei} for some recent computational 
attempts to find orthogonal Latin squares. 
The problem seems extremely hard, both from a 
combinatorial and from a computational point of 
view~\cite{kn:korner}. 
Despite many efforts, 
the existence of three pairwise orthogonal Latin squares of 
order 10 is, as far as we know, still unclear. 

\smallskip

Finally, the observations mentioned above lead us to conjecture 
that the edge-clique cover problem is NP-complete for cographs. 

\section{Planar graphs of bounded treewidth}

\begin{theorem}
Let $G$ be a planar graph with treewidth 
at most $k$. Then $\alpha^{\prime}(G)$ can be computed in 
$O(2^{O(k)}n)$ time, where $n$ is the number of vertices of $G$.
\end{theorem}
\begin{proof}
Let $\{\mathcal{X},T\}$ be a nice tree decomposition of $G=(V,E)$, 
where $\mathcal{X} = \{X_1,\dots,\\
X_t\}$ is a family of subsets of $V$
and $T$ is a tree whose nodes are the 
subsets $X_i$~\cite{kn:kloks2}.
Since $\{\mathcal{X},T\}$ is nice,
$T$ is a rooted tree and there are four types of node $i$.
We choose an arbitrary node $r$ in $T$ as the root of $T$.

\smallskip

\noindent
For each node $i$, $G(i)$ is a subgraph of $G$
formed by all nodes in sets $X_j$,
with $j=i$ or $j$ a descendant of $i$.
Let $E(X_i)$ denote a set of the edges induced by $X_i$.
For node $i$ in tree decomposition and a subset $F \subseteq E(X_i)$,
we define $\alpha^{\prime}(i,F)$
as the maximum cardinality of independent set $A$ of edges of $G(i)$
with $A \cap E(X_i) = F$.
Note that $\alpha^{\prime}(i,F)=-\infty$ if $A$ does not exist.

\smallskip

\noindent
In the following, we consider the four types of nodes in $T$
and calculate $\alpha^{\prime}(i,F)$.

\smallskip

\noindent
{\sc Leaf:}
Let $i$ be a leaf node with $|X_i|=1$. 
Then $F=\es$ and $\alpha^{\prime}(i,F) = 0$.

\smallskip

\noindent
{\sc Join:}
Let $i$ be a node with children $j_1$, $j_2$.
Then
\begin{equation}
\alpha^{\prime}(i, F) = 
\alpha^{\prime}(j_1, F) + \alpha^{\prime}(j_2, F)-|F|.
\end{equation}

\smallskip

\noindent
{\sc Introduce:}
Let $i$ be a node with child node $j$ such 
that $X_i = X_j \cup \{v\}$, for some vertex $v \in V$.
Let $F \subseteq E(X_j)$. 
Then $\alpha^{\prime}(i,F) = \alpha^{\prime}(j,F)$.
Let $F'$ be a subset of the edges in $E(X_i)$ that are incident with $v$.
Then if $F \cup F'$ is independent, 
\begin{equation}
\alpha^{\prime}(i,F \cup F') = 
\alpha^{\prime}(j,F)+|F'|.
\end{equation}
Otherwise, if $F$ is not independent then 
$\alpha^{\prime}(i,F \cup F')=-\infty$. 

\smallskip

\noindent
{\sc Forget:}
Let $i$ be a node with child node $j$ such that $X_i = X_j \setminus\{v\}$, for some vertex $v \in V$.
Let $F \subseteq E(X_i)$
and let $F'$ be a subset of the edges in $E(X_j)$ that are incident with $v$.
Then
\begin{equation}
\alpha^{\prime}(i,F) = \max\;\{\;\alpha^{\prime}(j,F),\alpha^{\prime}(j,F \cup F')\;\}.
\end{equation}

\smallskip

\noindent
For node $i$ in tree decomposition we compute a table with all values $\alpha^{\prime}(i,F)$
via dynamic programming.
Since $G$ is planar, each node $i$ contains $O(k)$ edges.
Thus for any type of nodes, a table contains $O(2^{O(k)})$ entries
and can be calculated in $O(2^{O(k)})$ time.
Therefore, $\alpha^{\prime}(G)=\max\;\{\alpha^{\prime}(r,F)\}$ can be obtained in $O(2^{O(k)}n)$ time.

\smallskip

\noindent
This completes the proof.
\qed\end{proof}

\section{The bibliography for the appendix}

\let\oldthebibliography=\thebibliography
\let\oldendthebibliography=\endthebibliography
\renewenvironment{thebibliography}[1]{%
 	\oldthebibliography{#1}%
    \setcounter{enumiv}{ 41 }%
}{\oldendthebibliography}

\end{document}